\documentclass[final,a4paper]{siamltex}

\usepackage{amsmath,amsfonts,amssymb,mathtools}
\usepackage{graphicx,subfigure}
\usepackage{braket,dsfont,url,enumerate}

\usepackage{pdfpages} 
\usepackage{cmap}
\usepackage{hyperref}
\usepackage{color,dsfont,braket}

\renewcommand{\u}{\vec{u}}

\newcommand{\x}{\vec{x}}
\newcommand{\z}{\vec{z}}

\newcommand{\V}{\vec{V}}
\newcommand{\w}{\vec{w}}

\renewcommand{\v}{\vec{v}}
\newcommand{\g}{\vec{g}}

\newcommand{\f}{\vec{f}}

\newcommand{\N}{\mathbb{N}}

\newcommand{\X}{\vec{X}}
\newcommand{\Th}{\mathcal{T}_h}

\newcommand{\Om}{\Omega}

\DeclarePairedDelimiter{\norm}{\lVert}{\rVert}

\DeclarePairedDelimiter{\twonorm}{\lVert}{\rVert_{L^2(\Omega)}}
\DeclarePairedDelimiter{\stwonorm}{\lVert}{\rVert^2_{L^2(\Omega)}}

\newcommand{\na}{\nabla}
\newcommand{\pa}{\partial}

\newcommand{\De}{\Delta}

\newcommand{\IOm}{I\times \Om}
\newcommand{\ImOm}{I_m\times \Om}

\newcommand{\vertiii}[1]{{\left\vert\kern-0.25ex\left\vert\kern-0.25ex\left\vert #1
    \right\vert\kern-0.25ex\right\vert\kern-0.25ex\right\vert}}

\newcommand{\Ppol}[1]{\mathcal{P}_{#1}}

\newtheorem{remark}[theorem]{Remark}

\begin{document}

\title{$L^2(I;H^1(\Omega))$ and $L^2(I;L^2(\Omega))$ best approximation type error estimates for Galerkin solutions of transient Stokes problems}

\author{
Dmitriy Leykekhman\footnotemark[2]
\and
Boris Vexler\footnotemark[3]
}

\pagestyle{myheadings}
\markboth{DMITRIY LEYKEKHMAN AND BORIS VEXLER}{Discrete maximal parabolic regularity}

\maketitle

\renewcommand{\thefootnote}{\fnsymbol{footnote}}
\footnotetext[2]{Department of Mathematics,
               University of Connecticut,
              Storrs,
              CT~06269, USA (dmitriy.leykekhman@uconn.edu). The author was partially supported by NSF grant DMS-1115288. }

\footnotetext[3]{Technical University of Munich,
Chair of Optimal Control,
Center for Mathematical Sciences,
Boltzmannstra{\ss}e 3,
85748 Garching by Munich, Germany, (vexler@ma.tum.de). }

\renewcommand{\thefootnote}{\arabic{footnote}}


\begin{abstract}
In this paper we establish best approximation type estimates for the fully discrete Galerkin solutions of transient Stokes problem in $L^2(I;L^2(\Omega)^d)$ and $L^2(I;H^1(\Omega)^d)$  norms. These estimates fill the gap in the error analysis of the transient Stokes problems and have a number of applications.  The analysis naturally extends to inhomogeneous parabolic problems.  The best type $L^2(I;H^1(\Omega))$ error estimates seems to be new even for  scalar parabolic problems. 
\end{abstract}



\section{Introduction}
In this paper we consider the following transient Stokes problem with no-slip boundary conditions, 
\begin{equation}\label{eq:transient:Stokes}
    \begin{aligned}
	\partial_t\u-\Delta \u+ \nabla p &= \f \quad &\text{in }I\times\Omega, \\
	\nabla \cdot \u &= 0 \quad &\text{in } I\times\Omega, \\
	\u &= \vec 0 \quad  &\text{on } I\times\partial \Omega, \\
	\u(0) &= \u_0 &\text{in } \Omega.
    \end{aligned}
\end{equation}
We assume that $\Omega\subset \mathbb{R}^d$, $d\in \{2,3\}$, is a bounded polygonal/polyhedral Lipschitz domain, $T>0$ and $I=(0,T]$.
In the next section we make precise assumptions on the data, which  allow for a weak formulation of the problem and provide regularity for the velocity   $\u \in L^2(I;H^1(\Omega)^d)$.
We consider fully discrete approximations of problem \eqref{eq:transient:Stokes}, where we use compatible finite elements (i.e. satisfying a uniform inf-sup condition) for the space discretization and the discontinuous Galerkin method for the temporal discretization. 
In our previous work \cite{Behringer_Leykekhman_Vexler_2022}, 
we have established the best type error estimate in ${L^{\infty}(I;L^2(\Omega))}$ norm,
\begin{equation}\label{eq:intro:best_approx}
    \norm{\u-\u_{{\tau}h}}_{L^{\infty}(I;L^2(\Omega))} \leq C\ell_\tau  \Big(\norm{\u-\vec{\chi}}_{L^{\infty}(I;L^2(\Omega))} +  \norm{\u - R_h^S(\u,p)}_{L^{\infty}(I;L^2(\Omega))} \Big), %
\end{equation}
where $\u_{{\tau}h}$ is the fully discrete finite element approximation of the velocity $\u$,  $R_h^S$ is the Ritz projection for the stationary Stokes problem, $\vec{\chi}$ is an arbitrary  function from the finite element approximation of the velocity spaces $X^w_{\tau}(\V_h)$, and $\ell_\tau$ is a logarithmic term. Such results are desirable for example in the analysis of PDE constrained optimal control problems since they do not require 
any additional regularity assumptions on the solution beyond the regularity which follows  directly from the  problem data. The estimate \eqref{eq:intro:best_approx} was an improvement of the main results in \cite[Thm.~4.9]{Chrysafinos_Walkington_2010}, where the error was estimated simultaneously  in ${L^{\infty}(I;L^2(\Omega))}$  and ${L^{2}(I;H^1(\Omega))}$ norms, and the bounds there required the presence of the "mixed terms".  A natural question, which actually was raised by one of the referee for \cite{Behringer_Leykekhman_Vexler_2022}, is it possible to obtain a best type approximation result just w.r.t. ${L^{2}(I;H^1(\Omega))}$ norm? Surprisingly, such results are not available even for scalar parabolic problems. In this note, we give positive answer to this question. In this paper we 
establish the following best type error estimates
\begin{equation}\label{eq:intro:best_approx L2}
    \norm{\u-\u_{{\tau}h}}_{L^2(\IOm)}  \le C  \Big(\norm{\u-\vec{\chi}}_{L^2(\IOm)} + \norm{\u-\pi_\tau \u}_{L^2(\IOm)}  +\norm{\u - R_h^S(\u,p)}_{L^2(\IOm)}  \Big)
\end{equation}
and 
\begin{equation}\label{eq:intro:best_approx H1}
\begin{aligned}
    \norm{\na(\u-\u_{{\tau}h})}_{L^2(\IOm)} \le C  \Big(\norm{\na(\u-\vec{\chi})}_{L^2(\IOm)} &+ \norm{\na(\u-\pi_\tau \u)}_{L^2(\IOm)}\\  &+\norm{\na(\u - R_h^S(\u,p))}_{L^2(\IOm)} \Big),
    \end{aligned}
\end{equation}
where as above $\u_{{\tau}h}$ is the fully discrete finite element approximation of the velocity $\u$, $\vec{\chi}$ is an arbitrary  function from the finite element approximation of the velocity spaces $X^w_{\tau}(\V_h)$, $\pi_\tau\u_{{\tau}h}$ is a certain time projection of the $\u$ on the time discrete space $X^w_{\tau}$, and  $R_h^S$ is the Ritz projection for the stationary Stokes problem.
The results \eqref{eq:intro:best_approx L2} and \eqref{eq:intro:best_approx H1} link the approximation error for the fully discrete transient Stokes problem to the best possible approximation of a continuous solution $\u$ in the discrete space $X^w_{\tau}(\V_h)$ as well as the approximation of the stationary Stokes problem in $\V_h$ and a time projection. Such results go in hand with only natural assumptions on the problem data and thus are desirable in  number of applications. The estimate \eqref{eq:intro:best_approx L2} does not require any additional regularity of the domain, thus allowing, e.g., for reentrant corners and edges. Moreover, \eqref{eq:intro:best_approx L2}  does not require the mesh to be quasi-uniform nor shape regular. Therefore, the result is also true for graded and even anisotropic meshes (provided the discrete inf-sup condition holds uniformly on such meshes). However, the second estimate \eqref{eq:intro:best_approx H1}, does require the stability of the discrete Leray projection in $H^1$ norm, which so far is established for the quasi-uniform meshes on convex domains. 
These results also naturally hold for the inhomogeneous heat equation, where the proofs can be simplified and extended to more general meshes, (see Section \ref{sec: parabolic}).

Under the additional assumption of convexity of $\Omega$ and some approximation properties of the discrete spaces, from \eqref{eq:intro:best_approx L2} and \eqref{eq:intro:best_approx H1}
we easily derive optimal error estimates of the form
\begin{equation}\label{eq: error in terms of data}
\norm{\u-\u_{{\tau}h}}_{L^{2}(I;H^s(\Omega))}\leq C \left({\tau}^{1-s/2} + h^{2-s}\right) \left( \norm{\f}_{L^{2}(I;L^2(\Omega))} + \norm{\u_0}_{\V^1}\right),\quad s=0,1,
\end{equation}
where $\V^1$ is an appropriate space introduced in the next section. This estimate is optimal  with respect to both the assumed regularity of the data and the order of convergence.

The above results naturally hold for simpler case of inhomogeneous heat equation with straightforward change of operators and function spaces. The 
$L^{2}(I;L^2(\Omega))$ best approximation result is essentially shown in \cite{MeidnerVexler:2008I}, instead of the best approximation the optimal error estimate in terms of data of the form  \eqref{eq: error in terms of data} is given, however, the 
$L^{2}(I;H^1(\Omega))$ best approximation result seems to be new.

The rest of the paper is structured as follows. In Section \ref{sec: continuous}, we introduce the function spaces, key operators and weak formulations of the problem with regularity results. In Section \ref{sec: discretization}, we introduce fully discrete Galerkin approximation of the problem. Section \ref{sec: Discrete results} is devoted to stability results of the fully discrete velocity solutions. In Section \ref{sec: Discrete results} we review some stability and approximation results for the stationery Stokes problem, which we requite for our main results Theorem \ref{chap:IS:theorem:L2_best_approximation}  and Theorem \ref{thm: H1 best approximation} in the next Section \ref{sec: main results}. Finally, in Section \ref{sec: parabolic}, we briefly discuss the extension of the main results to scalar parabolic problems.


\section{Continuous problem} \label{sec: continuous}

In this section, we introduce function spaces we require for the analysis of \eqref{eq:transient:Stokes} and state some of the main properties of these spaces.   In our presentation  we follow the notation and presentation of \cite[Section~1 and Section~2]{Guermond_2008}.

\subsection{Function spaces and Stokes operator}
In the following,
we will use the usual notation to denote the Lebesgue spaces $L^p$ and Sobolev spaces $H^k$ and $W^{k,p}$. The space  $L^2_0(\Om)$ will denote a subspace of $L^2(\Om)$  with mean-zero functions.  The inner product on $L^2(\Omega)$ as well as on $L^2(\Omega)^d$ is denoted by $(\cdot,\cdot)$. To improve readability, we omit the superscript $d$ when having for example $L^2(\Omega)^d$ appear as subscript to norms.
We also introduce 
the following function spaces
\begin{equation}
    \mathcal{V} = \Set{\v \in C^{\infty}_0(\Omega)^d | \nabla \cdot \v =0},\quad
    \V^0 =  \overline{\mathcal{V}}^{L^2},\quad    \V^1 =  \overline{\mathcal{V}}^{H^{1}},%
\end{equation}
where the notation in the last line denotes the completion of the space $\mathcal{V}$ with respect to the $L^2(\Omega)^d$ and $H^1(\Omega)^d$ topology, respectively. Notice that functions in $\V^1$ %
have zero boundary conditions in the trace sense.
Alternatively we have  
$$
\V^1= \set{\v \in H_0^{1}(\Om)^d | \nabla \cdot \v = 0}
$$
by \cite[Theorem III.4.1]{2011Galdi}.

We define the vector-valued Laplace operator
$$
-\Delta \colon D(\Delta) \rightarrow L^2(\Omega)^d,
$$
where the domain $D(\Delta)$ is understood with respect  to $L^2(\Omega)^d$ and is given as
\[
    D(\Delta) = \set{\v \in H_0^{1}(\Om)^d | \Delta \v \in L^2(\Omega)}.
\]
If the domain $\Omega$ is  convex, then the standard $H^2(\Omega)$ regularity implies $ D(\Delta) = H^1_0(\Omega)^d \cap H^2(\Omega)^d$. 
In addition, we introduce the space $\V^2$ as
\[
    \V^2 = \V^1 \cap D(\Delta).
\]
We will also use the following Helmholtz decomposition (cf. \cite[Chapter I, Theorem 1.4]{1977Temam} and \cite[Theorem III.1.1]{2011Galdi})
\begin{equation}
    L^2(\Omega)^d = \V^0 \oplus \nabla \left(H^{1}(\Omega) \cap L^2_0(\Omega)\right). \label{chap:IS:eq:helmholtz_decomposition}
\end{equation}
As usual we define the Helmholtz projection $\mathbb{P} \colon L^2(\Omega)^d \rightarrow \V^{0}$ (often also called the Leray  projection) as the $L^2$-projection from $L^2(\Om)^d$ onto $\V^{0}$. %
Using $\mathbb{P}$ and $-\Delta$, we define the Stokes operator $A\colon \V^2 \rightarrow \V^0$ as
\begin{equation}\label{eq:StokesOperator}
    A = - \mathbb{P} \Delta \vert_{\V^2}.
\end{equation}
The operator $A$ is a self adjoint, densely defined and positive definite operator on $\V^0$. We note that $D(A)=\V^2$. %
Similar to the Laplace operator, for convex polyhedral domains $\Om$ we have the following $H^2$ regularity bound due to
\cite{1989Dauge,1976Kellogg}
\begin{equation}\label{eq: H2 regularity A}
    \|\vec{v}\|_{H^2(\Om)}\le C\|A\vec{v}\|_{L^2(\Om)},\quad \forall \vec{v}\in \V^2.
\end{equation}

\subsection{Weak formulation and regularity}
In this section we discuss the weak formulation and the regularity  of the transient Stokes problem \eqref{eq:transient:Stokes}. We will use the notation $L^s(I;X)$ for the corresponding Bochner space with a Banach space $X$. Moreover, we will use also the standard notation $H^1(I;X)$. The inner product in $L^2(I;L^2(\Omega)^d)$  is denoted by $(\cdot,\cdot)_{I \times \Omega}$. We will also use the notation $(f,g)_{I \times \Omega}$ for the corresponding integral for $f \in L^2(I;L^2(\Omega)^d)$ and $g \in L^{2}(I;L^2(\Omega)^d)$.

\begin{proposition}\label{prop:weak_solution}
    Let $\f \in L^2(I;(\V^1)')$ and $\u_0 \in \V^0$. Then there exists a unique solution $\u \in L^2(I;\V^1) \cap C(\bar I,\V^0)$ with $\partial_t \u \in  L^2(I;(\V^1)')$ fulfilling
    $\u(0) = \u_0$ and
    \begin{equation}\label{eq: weak stokes}
	\langle \partial_t \u,\v\rangle + (\nabla \u,\nabla \v)_{I \times \Omega} = ( \f,\v)_{I \times \Omega} \quad \text{for all } \v \in L^2(I;\V^1).
    \end{equation}
    There holds 
    \[
	\norm{\nabla \u}_{L^2(I;L^2(\Omega))} +\norm{\u'}_{L^2(\bar I;(\V^1)')}+ \norm{\u}_{C(\bar I;L^2(\Omega))} \le C \left(\norm{\f}_{L^2(I;(\V^1)')} + \norm{\u_0}_{\V^0}\right).
    \]
\end{proposition}
\begin{proof}
    For the proof of the above result we refer to \cite[Chapter III, Theorem 1.1]{1977Temam}.
\end{proof}
It is well known, cf. again \cite{1977Temam}, that  the equation \eqref{eq: weak stokes} can be understood as an abstract parabolic problem
\begin{equation}\label{eq:stokes_abstract_equation}
    \begin{aligned}
	\partial_t  \u + A \u &= \mathbb{P}\f\qquad \text{for a.a. } t \in I,\\
	\u(0) &= \u_0,
    \end{aligned}	
\end{equation}
with the Stokes operator $A$ defined in \eqref{eq:StokesOperator}.

The next theorem provides the space-time weak formulation in both variables, velocity and pressure. Please note, that no additional regularity of the domain is required.  
\begin{theorem}\label{theorem:weak_with_preasure}
    Let $\f \in L^2(I;L^2(\Omega)^d)$  and $\u_0 \in \V^1$. 
    Then there exists a unique solution $(\u,p)$ with
    \[
	\u \in L^2(I;\V^1), \; \partial_t \u, A \u \in L^2(I; L^2(\Omega)^d) \quad \text{and} \quad p \in L^2(I;L^2_0(\Omega))
    \]
    fulfilling $\u(0)=\u_0$ and
    \begin{equation}\label{eq: weak stokes with pressure}
	(\partial_t \u,\v)_{I \times \Omega} + (\nabla \u,\nabla \v)_{I \times \Omega} - (p,\nabla \cdot \v)_{I \times \Omega} + (\nabla \cdot \u,\xi)_{I \times \Omega}= (\f,\v)_{I \times \Omega}
    \end{equation}
    for all
    \[
	\v \in L^2(I;H^1_0(\Omega)^d)\quad \text{and}\quad \xi \in L^2(I;L^2_0(\Omega)).
    \]
    There holds the estimate
    \[
  \norm{\partial_t \u}_{L^2(I;L^2(\Omega))}+\norm{A \u }_{L^2(I;L^2(\Omega))}+  \norm{p}_{L^2(I;L^2(\Omega))} \le C \left( \norm{\vec{f}}_{L^2(I;L^2(\Omega))} + \norm{\u_0}_{\V^1}\right).
    \]
\end{theorem}
\begin{proof}
The proof is given in \cite{Behringer_Leykekhman_Vexler_2022}, Theorem 2.4.
\end{proof}

\begin{corollary}\label{cor:Omega_convex_2}
	Let the assumptions of Theorem \ref{theorem:weak_with_preasure} hold and let in addition the domain $\Omega$ be convex. Then we have $\u \in L^2(I;H^2(\Omega)^d) $ and $p \in L^2(I, H^1(\Omega))$ with the corresponding estimates
    \[
   \norm{ \u}_{L^2(I;H^2(\Omega))}+ \norm{p}_{L^2(I;H^1(\Omega))} \le C\left(\norm{\vec{f}}_{L^2(I;L^2(\Omega))} + \norm{\u_0}_{\V^1}\right).
    \]
\end{corollary}
\begin{proof}
This result is shown in \cite{1977Temam}, Proposition 1.2 in Chapter 3,   for $C^2$ domains, but the proof is valid for convex domains as well. See also Corollary 2.1 in \cite{Behringer_Leykekhman_Vexler_2022}.
\end{proof}

\section{Fully discrete discretization}\label{sec: discretization}
In this section we consider the  discrete version of the operators presented in the previous section and introduce fully discrete solution Galerkin solution.  

\subsection{Spatial discretization}\label{sec: space discretization}

Let $\{\Th\}$\label{glos:triangulation} be a 
family of triangulations of $\bar \Omega$, consisting of closed simplices, where we denote by $h$\label{glos:h} the maximum mesh-size. Let $\X_h \subset H^1_0(\Omega)^d$\label{glos:fe_space_velocity} and $M_h \subset L^2_0(\Omega)$\label{glos:fe_space_pressure} be a pair of compatible finite element spaces, i.e., them satisfying a uniform discrete inf-sup condition,
\begin{equation}\label{chap02:eq:discrete_infsup}
    \sup_{\v_h \in \X_h}\frac{(q_h,\nabla \cdot \v_h) }{\twonorm{\nabla \v_h}} \geq \beta \twonorm{q_h} \quad \forall q_h \in M_h, 
\end{equation}%
with a constant $\beta >0$ independent of $h$. 

We introduce the usual discrete Laplace operator $-\Delta_h \colon \X_h \rightarrow \X_h$ by
\begin{equation}
    (-\Delta_h \z_h, \v_h) = (\nabla \z_h, \nabla \v_h), \qquad \forall \z_h,\v_h \in \X_h.
\end{equation}
To define a discrete version of the Stokes operator $A$, we first define the space of discretely divergence-free vectors $\V_h$ as
\begin{equation}\label{eq:discretely_divergence_Xh}
    \V_h = \Set{ \v_h \in \X_h | (\nabla \cdot \v_h, q_h) =0 \quad \forall q_h \in M_h}.
\end{equation}
Using this space we can define the discrete Leray projection $\mathbb{P}_h\colon L^1(\Om)^d\to \V_h$ to be the $L^2$-projection onto $\V_h$, i.e., 
\begin{equation}\label{eq:discrete_Leray_projection}
(\mathbb{P}_h \u, \v_h)=(\u,\v_h) \quad\forall \v_h \in \V_h.
\end{equation}
Using $\mathbb{P}_h$, we define the discrete Stokes operator $A_h\colon \V_h \rightarrow \V_h$ as $A_h = - \mathbb{P}_h \Delta_h\vert_{\V_h}$. By this definition we have that for $\u_h \in \V_h$, $A_h\u_h \in \V_h$ fulfills
\begin{equation}
    (A_h \u_h, \v_h) = (\nabla \u_h, \nabla \v_h), \qquad \forall \v_h \in \V_h.
\end{equation}
Notice, since $\V_h \subset \X_h$,   for $\v_h \in \V_h$ we obtain
\begin{equation}\label{chap:IS:eigenvalue}
    (A_h \v_h, \v_h) = (\nabla \v_h, \nabla \v_h)  \geq \lambda_0 \stwonorm{\v_h},
\end{equation}
where $\lambda_0$ is the smallest eigenvalue of $-\Delta$.
This implies that the eigenvalues of $A_h$ are also positive and bounded from below by $\lambda_0$.

Moreover we define the orthogonal space $\V_h^\perp \subset \X_h$ as
\[
\V_h^\perp = \Set{\w_h \in X_h | (\w_h,\v_h) = 0 \quad \forall \v_h \in \V_h}.
\]
The following classical result, cf., e.g., \cite[Chapter II, Theorem 1.1]{1986Girault} will be used to provide existence and uniqueness of the fully discrete pressure in the sequel.
\begin{lemma}\label{lemma:stationary_existence_pressure}
For every $\w_h \in \V_h^\perp$ there exists a unique $p_h \in M_h$ such that
\[
(\w_h,\v_h) = (p_h, \nabla \cdot \v_h) \quad \forall \v_h \in \X_h.
\]
There holds
\[
\norm{p_h}_{L^2(\Omega)} \le \frac{1}{\beta} \norm{\nabla (-\Delta_h)^{-1} \w_h}_{L^2(\Omega)}.
\]
\end{lemma}

\subsection{Temporal discretization: the discontinuous Galerkin method}
\label{sec: time discretization}
\leavevmode \\
In this section we introduce  the discontinuous Galerkin method for the time discretization of the transient Stokes equations, a similar method was considered, e.g., in \cite{Chrysafinos_Walkington_2010} and \cite{Behringer_Leykekhman_Vexler_2022}.
For that, we partition $I = (0,T]$ into subintervals $I_m = (t_{m-1},t_m]$ of length ${\tau}_m = t_m - t_{m-1}$, where $0= t_0<t_1<\dots <t_{M-1}<t_M = T$. The maximal and minimal time steps are denoted by ${\tau}=\max_m {\tau}_m$ and ${\tau}_{\min}=\min_m{\tau}_m$, respectively.  The time partition fulfills the following assumptions:
\begin{enumerate}
    \item There are constants $C,\beta >0$ independent of ${\tau}$ such that
	\begin{equation}
	    {\tau}_{\min} \geq C {\tau}^{\beta}. 
	\end{equation}
    \item There is a constant $\kappa >0$ independent of ${\tau}$ such that for all\\$m=1, 2, \dots, M-1$
	\begin{equation}
	    \kappa^{-1} \leq \frac{{\tau}_m}{{\tau}_{m+1}} \leq \kappa.
	\end{equation}
    \item It holds ${\tau} \leq \frac{T}{4}$.
\end{enumerate}
For a given Banach space $\mathcal{B}$ and the order $w \in \N$ we define the semi-discrete space $X_{\tau}^w(\mathcal{B})$ of piecewise polynomial functions in time as
\begin{equation}
    X_{\tau}^w(\mathcal{B}) = \Set{ \v_{\tau} \in L^2(I; \mathcal{B}) | \v_{\tau}\vert_{I_m} \in \mathcal{P}_{w,I_m}(\mathcal{B}), m = 1,2, \dots, M},
\end{equation}
where $\mathcal{P}_{w,I_m}(\mathcal{B})$ is the space of polynomial functions of degree less or equal $w$ in time with values in $\mathcal{B}$, i.e.,
\begin{equation}\label{chap:IS:eq:polynomial}
    \mathcal{P}_{w,I_m}(\mathcal{B}) = \Set{ \v_{\tau} \colon I_m \rightarrow \mathcal{B} | \v_{\tau}(t) = \sum_{j=0}^w \v^j\phi_j(t) ,\;\v^j \in \mathcal{B}, j =0, \dots, w}. 
\end{equation}
Here, $\{\phi_j(t)\}$ is a polynomial basis in $t$ of the space $\mathcal{P}_w(I_m)$ of polynomials with degree less or equal $w$ over the interval $I_m$.
We use the following standard notation for a function $\u \in X_{\tau}^w(L^2(\Omega)^d)$%
\begin{equation}
\u_m^+ = \lim_{\varepsilon \rightarrow 0^+}\u(t_m+\varepsilon), \quad \u_m^- = \lim_{\varepsilon \rightarrow 0^+}\u(t_m-\varepsilon),\quad [\u]_m = \u_m^+-\u_m^-.
\end{equation}
We define the bilinear form $\mathfrak{B}$ by
    \[
	\mathfrak{B}(\u,\v) = \sum_{m=1}^M (\partial_t\u, \v )_{I_m\times \Omega} + (\nabla \u, \nabla \v)_{I\times \Omega}
	+ \sum_{m=2}^M ([\u]_{m-1},\v^+_{m-1})_{\Omega} + (\u_0^+,\v_0^+)_{\Omega}.
    \]
    With this bilinear form we define the fully discrete approximation for the transient Stokes problem on the discrete divergence free space $X_{\tau}^w(\V_h)$:
    \begin{equation}\label{eq:fully_discrete_div_free}
	\u_{\tau h} \in X_{\tau}^w(\V_h) \;:\; \mathfrak{B}(\u_{\tau h},\v_{\tau h}) = (\f,\v_{\tau h})_{I \times \Omega} + (\u_0, \v_{\tau h,0}^+)_{\Omega} \quad \forall \v_{\tau h} \in X_{\tau}^w(\V_h).
    \end{equation}
    By a standard argument one can see that this formulation possesses a unique solution (existence follows from uniqueness by the fact that \eqref{eq:fully_discrete_div_free} is equivalent to a quadratic system of linear equations). 
 \begin{remark}\label{remark_on_Ph}
    Note, that the data $\f$ and $\u_0$ in \eqref{eq:fully_discrete_div_free} can be replaced by $\mathbb{P}_h \f$ and $\mathbb{P}_h \u_0$ respectively (with $\mathbb{P}_h$ being the discrete Leray projection \eqref{eq:discrete_Leray_projection})  without changing the solution. 
\end{remark}

The above formulation is not a conforming discretization of the divergence free formulation \eqref{eq: weak stokes} due to the fact that $X_{\tau}^w(\V_h)$ is not a subspace of $L^2(I;\V^1)$. In order to introduce a velocity-pressure discrete formulation (as a discretization of \eqref{eq: weak stokes with pressure}) we consider the following bilinear form
\begin{multline}
    B((\u,p),(\v,q)) = \sum_{m=1}^M ( \partial_t\u, \v )_{I_m\times \Omega} + (\nabla \u, \nabla \v)_{I\times \Omega} - (p,\nabla \cdot \v)_{I\times\Omega} + (\nabla \cdot \u,q)_{I\times\Omega}\\
    + \sum_{m=2}^M ([\u]_{m-1},\v^+_{m-1})_{\Omega} + (\u_0^+,\v_0^+)_{\Omega}.
\end{multline}
The corresponding fully discrete formulation reads: find $(\u_{{\tau}h},p_{{\tau}h})\in X^w_{\tau}(\X_h \times M_h)$ such that
\begin{equation}\label{eq:spacetime_discretization}
   B((\u_{{\tau}h},p_{{\tau}h}),(\v_{{\tau}h},q_{{\tau}h})) = (\f,\v_{{\tau}h})_{I\times \Omega} + (\u_0, \v_{{\tau}h,0}^+)_{\Omega} \qquad \forall (\v_{{\tau}h},q_{{\tau}h}) \in X^w_{\tau}(\X_h \times M_h).
\end{equation}
We note, that for the temporal discretization we use polynomials of the same order for the velocity and pressure. The next proposition states the equivalence of the formulation \eqref{eq:fully_discrete_div_free} and \eqref{eq:spacetime_discretization}.
\begin{proposition}
For a solution $(\u_{{\tau}h},p_{{\tau}h})$ of \eqref{eq:spacetime_discretization} the discrete velocity $\u_{{\tau}h}$ fulfills \eqref{eq:fully_discrete_div_free}. Moreover, for a solution $\u_{{\tau}h}$ of \eqref{eq:fully_discrete_div_free} there exists a unique $p_{{\tau}h} \in X^w_{\tau}(M_h)$ such that the pair $(\u_{{\tau}h},p_{{\tau}h})$ fulfills \eqref{eq:spacetime_discretization}. In particular the solution of \eqref{eq:spacetime_discretization} is unique.
\end{proposition}
\begin{proof}
The proof is given in \cite{Behringer_Leykekhman_Vexler_2022}, Proposition 4.1.
\end{proof}

The next proposition provides the Galerkin orthogonality relation for the velocity pressure discretization \eqref{eq:spacetime_discretization}, which is essential for our analysis. Please note, that for the velocity formulation \eqref{eq:fully_discrete_div_free} the Galerkin orthogonality does not hold due to the fact that $X_{\tau}^w(\V_h)$ is not a subspace of $L^2(I,\V^1)$.
\begin{proposition}\label{prop:Galerkin_orthogonality}
Let the assumptions of \eqref{theorem:weak_with_preasure} be fulfilled, i.e., $\f \in L^2(I;L^2(\Omega)^d)$  and $\u_0 \in \V^0$. Then there holds for the solution $(\u,p)$ of \eqref{eq: weak stokes with pressure}
\[
B((\u,p),(\v_{\tau h},q_{\tau h})) =  (\f,\v_{{\tau}h})_{I\times \Omega} + (\u_0, \v_{{\tau}h,0}^+)_{\Omega} \qquad \forall (\v_{{\tau}h},q_{{\tau}h}) \in X^w_{\tau}(\X_h \times M_h)
\]
and consequently 
\begin{equation}\label{eq:Galerkin_orthogonality}
B((\u-\u_{\tau h},p-p_{\tau h}),(\v_{\tau h},q_{\tau h})) = 0 \qquad \forall (\v_{{\tau}h},q_{{\tau}h}) \in X^w_{\tau}(\X_h \times M_h).
\end{equation}
\end{proposition} 
\begin{proof}
The proof is given in \cite{Behringer_Leykekhman_Vexler_2022}, Proposition 4.2.
\end{proof}

In the following, we also consider a dual problem, where we use a dual representation of the bilinear form $B$
\begin{equation}\label{eq:dualB}
\begin{aligned}
    B((\u,p),(\v,q)) = -&\sum_{m=1}^M \langle \u, \partial_t\v \rangle_{I_m\times \Omega} + (\nabla \u, \nabla \v)_{I\times \Omega} - (p,\nabla \cdot \v)_{I\times\Omega}\\ &+ (\nabla \cdot \u,q)_{I\times\Omega}
    - \sum_{m=1}^{M-1} (\u^{-}_{m},[\v]_{m})_{\Omega} + (\u_M^-,\v_M^-)_{\Omega},
\end{aligned}
\end{equation}
which is obtained by integration by parts and rearranging the terms in the sum.

We will also need the following projection $\pi_\tau$ for $v \in C(I,L^2(\Omega))$ with $\pi_\tau v|_{I_m} \in P_q(L^2(\Omega))$ for $m=1,2,\dots,M$ on each subinterval $I_m$ by
\begin{equation}\label{eq: projection pi_k}
\begin{aligned}
(\pi_\tau v-v,\phi)_{I_m\times \Omega}&=0,\quad \forall \phi\in P_{q-1}(I_m,L^2(\Omega)),\quad q>0,\\
\pi_\tau v(t_m^-)=v(t_m^-).
\end{aligned}
\end{equation}
In the case $q = 0$, $\pi_\tau v$ is defined solely by the second condition.
The following approximation property also hold
\begin{equation}\label{eq: approximation of pi_k}
\|\pi_\tau u-u\|_{L^2(I_m)}\le C\tau \|\partial_t u\|_{L^2(I_m)}\, \quad \forall u\in H^1(I_m).
\end{equation}

\section{Fully discrete stability results}\label{sec: Discrete results}
In this section, we establish stability result, which are essential for our main result. 
\begin{theorem}\label{thm: stability in for Ah}
For $\f \in L^2(I,L^2(\Omega)^d)$ and $\u_0 \in \V^1$. Let $\u_{{\tau}h}\in X^w_{\tau}(\V_h)$ be the solution to \eqref{eq:fully_discrete_div_free}. Then  there holds
\begin{multline}\label{eq:dmpr_s2}
	\Bigg( \sum_{m=1}^M \norm{\partial_t \u_{{\tau}h}}^2_{L^2(\ImOm)}\Bigg)^{1/2} + \norm{A_h \u_{{\tau}h}}_{L^2(I;L^2(\Omega))}\\%
	+ \Bigg( \sum^M_{m=1} {\tau}_m \norm{{\tau}_m^{-1} [\u_{{\tau}h}]_{m-1}}^2_{L^2(\Omega)} \Bigg)^{1/2} 
	\leq C  \left(\norm{\mathbb{P}_h\f}_{L^2(I; L^2(\Omega))}+\norm{\na\mathbb{P}_h\u_0}_{L^2(\Omega)}\right).
\end{multline}
\end{theorem}
\begin{proof}
The proof follows the lines of the corresponding proof in \cite{MeidnerVexler:2008I} replacing $-\Delta_h$ by $A_h$.
\end{proof}

For the $L^2(I; H^1(\Om)^d)$ norm estimates we also need the following result.
\begin{theorem}\label{thm: stability in for grad_Ah}
For $\f \in L^2(I,L^2(\Omega)^d)$ and $\u_0 \in \V^0$. Let $\u_{{\tau}h}\in X^w_{\tau}(\V_h)$ be the solution to \eqref{eq:fully_discrete_div_free}. Then  there holds
\begin{multline}\label{eq:dmpr_s3}
	\Bigg( \sum_{m=1}^M \norm{\partial_t \na A_h^{-1}\u_{{\tau}h}}^2_{L^2(I_m; L^2(\Omega))}\Bigg)^{1/2} + \norm{\nabla \u_{{\tau}h}}_{L^2(I;L^2(\Omega))}\\%
	+ \Bigg( \sum^M_{m=1} {\tau}_m \norm{{\tau}_m^{-1} [\na A_h^{-1}\u_{{\tau}h}]_{m-1}}^2_{L^2(\Omega)} \Bigg)^{1/2} 
	\leq C \left( \norm{\na A_h^{-1}\mathbb{P}_h\f}_{L^2(I; L^2(\Omega))}+\norm{\mathbb{P}_h\u_0}_{L^2(\Omega)}\right).
\end{multline}
\end{theorem}
\begin{proof}
The proof goes along the lines of the  proof of Theorem 4.1 \cite{MeidnerVexler:2008I}. We will provide the proof for completeness.

Testing \eqref{eq:fully_discrete_div_free} with $\v_{\tau h}=\u_{\tau h}$ we obtain on each time interval $I_m$
$$
(\partial_t \u_{\tau h},\u_{\tau h})_{\ImOm}+\|\na \u_{\tau h}\|^2_{L^2(\ImOm)}+([\u_{\tau h}]_{m-1},\u_{\tau h,m-1}^+)_{\Omega}=(\f,\u_{\tau h})_{\ImOm}.
$$
Using the identities
$$
(\partial_t \u_{\tau h},\u_{\tau h})_{\ImOm} = \frac{1}{2}\| \u_{\tau h,m}^-\|^2_{L^2(\Om)}-\frac{1}{2}\| \u_{\tau h,m-1}^+\|^2_{L^2(\Om)}
$$
and
$$
([\u_{\tau h}]_{m-1},\u_{\tau h,m-1}^+)_{\Omega}= \frac{1}{2}\| \u_{\tau h,m-1}^+\|^2_{L^2(\Om)}+\frac{1}{2}\| [\u_{\tau h}]_{m-1}\|^2_{L^2(\Om)}-\frac{1}{2}\| \u_{\tau h,m-1}^-\|^2_{L^2(\Om)}
$$
gives us
$$
\frac{1}{2}\| \u_{\tau h,m}^-\|^2_{L^2(\Om)}-\frac{1}{2}\| \u_{\tau h,m-1}^-\|^2_{L^2(\Om)}+\frac{1}{2}\| [\u_{\tau h}]_{m-1}\|^2_{L^2(\Om)}+\|\na \u_{\tau h}\|^2_{L^2(\ImOm)}=(\f,\u_{\tau h})_{\ImOm}.
$$
Summing over $m$, we obtain
$$
\frac{1}{2}\| \u_{\tau h,M}^-\|^2_{L^2(\Om)}+\frac{1}{2}\sum_{m=1}^M\| [\u_{\tau h}]_{m-1}\|^2_{L^2(\Om)}+\|\na \u_{\tau h}\|^2_{L^2(\IOm)}=(\f,\u_{\tau h})_{\IOm}+\frac{1}{2}\| \mathbb{P}_h\u_0\|^2_{L^2(\Om)}.
$$
To treat the term involving $\f$, we write it as
$$
(\f,\u_{\tau h})_{\IOm}= (\mathbb{P}_h\f,\u_{\tau h})_{\IOm}=(A_hA_h^{-1}\mathbb{P}_h\f,\u_{\tau h})_{\IOm}=(\na A_h^{-1}\mathbb{P}_h\f,\na \u_{\tau h})_{\IOm},
$$
where in the last step we used that $\u_{\tau h}$ is discrete divergence free. 
Using the Cauchy-Schwarz  and geometric-arithmetic mean inequalities
$$
(\f,\u_{\tau h})_{\IOm} \le \frac{1}{2}\|\na \u_{\tau h}\|^2_{L^2(\IOm)}+\frac{1}{2}\|\na A_h^{-1}\mathbb{P}_h\f\|^2_{L^2(\IOm)}.
$$
Thus, we obtain
\begin{equation}\label{eq: estimate for grad}
\|\na \u_{\tau h}\|^2_{L^2(\IOm)}\le \|\na A_h^{-1} \mathbb{P}_h\f\|^2_{L^2(\IOm)}+\|\mathbb{P}_h\u_0\|^2_{L^2(\Om)}.
\end{equation}
Next, we test \eqref{eq:fully_discrete_div_free}   with $\v_{\tau h}\mid_{I_m}=(t-t_{m-1})A_h^{-1} \partial_t \u_{\tau h}$. Noticing that  the jump terms disappear and using the identity,
$$
(\partial_t \u_{\tau h}, A_h^{-1} \partial_t u_{kh})_{\Om}=(A_h A_h^{-1} \partial_t \u_{\tau h}, A_h^{-1} \partial_t \u_{\tau h})_{\Om}=\|\na A_h^{-1} \partial_t \u_{\tau h}\|^2_{L^2(\Om)},
$$
 on each time interval $I_m$, which follows from $\partial_t \u_{\tau h}$ being discretely divergence free for each $t\in I_m$, we obtain
$$
\begin{aligned}
\int_{I_m}(t-t_{m-1})\|\na A_h^{-1} \partial_t \u_{\tau h}\|^2_{L^2(\Om)}dt=&\int_{I_m}(t-t_{m-1})(\na \u_{\tau h}, \na A_h^{-1} \partial_t \u_{\tau h})_{\Om}dt\\
&+\int_{I_m}(t-t_{m-1})(\f,A_h^{-1} \partial_t \u_{\tau h})_{\Om}dt\\
=&\int_{I_m}(t-t_{m-1})(\na \u_{\tau h}+\na A_h^{-1}\mathbb{P}_h \f, \na A_h^{-1} \partial_t \u_{\tau h})_{\Om}dt.
\end{aligned}
$$
Using the Cauchy-Schwarz inequality
$$
\begin{aligned}
\int_{I_m}(t-t_{m-1})(\na \u_{\tau h}+\na A_h^{-1}\mathbb{P}_h \f, \na A_h^{-1} \partial_t \u_{\tau h}&)_{\Om}dt
\le  \left(\int_{I_m}(t-t_{m-1})\|\na A_h^{-1} \partial_t \u_{\tau h}\|^2_{L^2(\Om)}dt\right)^{1/2}\\
&
\left(\int_{I_m}(t-t_{m-1})\|\na \u_{\tau h}+\na A_h^{-1}\mathbb{P}_h \f\|^2_{L^2(\Om)}dt\right)^{1/2}.
\end{aligned}
$$
Hence, canceling we have
$$
\int_{I_m}(t-t_{m-1})\|\na A_h^{-1} \partial_t \u_{\tau h}\|^2_{L^2(\Om)}dt\le \int_{I_m}(t-t_{m-1})\|\na \u_{\tau h}+\na A_h^{-1}\mathbb{P}_h \f\|^2_{L^2(\Om)}dt.
$$
Using the following inequality 
\begin{equation}\label{Inverse_Holder_type_inequality}
\int_{I_m}\|v_{k}\|^2dt\le C\tau_m^{-1}\int_{I_m}(t-t_{m-1})\|v_{k}\|^2dt \quad \forall v_k\in \Ppol{q}(I_m;\V_h),
\end{equation}
which can be easily verified by a usual scaling argument,
we obtain
$$
\begin{aligned}
\int_{I_m}\|\na A_h^{-1} \partial_t \u_{\tau h}\|^2_{L^2(\Om)}dt&\le C\tau_{m}^{-1}\int_{I_m}(t-t_{m-1})\|\na A_h^{-1} \partial_t \u_{\tau h}\|^2_{L^2(\Om)}dt\\
&\le C\tau_m^{-1}\int_{I_m}(t-t_{m-1})\|\na \u_{\tau h}+\na A_h^{-1}\mathbb{P}_h \f\|^2_{L^2(\Om)}dt\\
&\le C\int_{I_m}\|\na \u_{\tau h}+\na A_h^{-1}\mathbb{P}_h \f\|^2_{L^2(\Om)}dt.
\end{aligned}
$$
Summing over $m$ and using the estimate for the gradient \eqref{eq: estimate for grad}, we obtain
\begin{equation}\label{eq: estimate for partial_t Delta}
\sum_{m=1}^M\|\na A_h^{-1}\pa_t \u_{\tau h}\|^2_{L^2(\ImOm)}\le C\left( \|\na A_h^{-1} \mathbb{P}_h\f\|^2_{L^2(\IOm)}+\|\mathbb{P}_h\u_0\|^2_{L^2(\Om)}\right).
\end{equation}
It remains to estimate the jump terms. We test \eqref{eq:fully_discrete_div_free}  with $\v_{\tau h}\mid_{I_m}= [A_h^{-1}\u_{kh}]_{m-1}$.
On each time interval $I_m$ this time we obtain,
$$
\|\na  [A_h^{-1}\u_{\tau h}]_{m-1}\|^2_{L^2(\Om)}=(\f+A_h \u_{\tau h}+\partial_t \u_{\tau h}, [A_h^{-1} \u_{\tau h}]_{m-1})_{\ImOm}.
$$
Integrating by parts and using the Cauchy-Schwarz  and geometric-arithmetic mean inequalities
$$
\begin{aligned}
(&\f+A_h \u_{\tau h}+\partial_t \u_{\tau h}, [A_h^{-1} \u_{\tau h}]_{m-1})_{\ImOm}\\
&=(\na A_h^{-1}\mathbb{P}_h\f+\na \u_{\tau h}+\na A_h^{-1}\partial_t \u_{\tau h},[\na A_h^{-1}\u_{\tau h}]_{m-1})_{\ImOm}\\
&\le \|\na A_h^{-1}\mathbb{P}_hf+\na \u_{\tau h}+\na A_h^{-1}\partial_t \u_{\tau h}\|_{L^2(\ImOm)}\|[\na A_h^{-1} \u_{\tau h}]_{m-1}\|_{L^2(\ImOm)}\\
&\le \frac{\tau_m}{2}\|\na A_h^{-1}\mathbb{P}_h\f+\na \u_{\tau h}+\na A_h^{-1}\partial_t \u_{\tau h}\|^2_{L^2(\ImOm)}+\frac{1}{2\tau_m}\|[\na A_h^{-1} \u_{\tau h}]_{m-1}\|^2_{L^2(\ImOm)}\\
&= \frac{\tau_m}{2}\|\na A_h^{-1}\mathbb{P}_hf+\na \u_{\tau h}+\na A_h^{-1}\partial_t \u_{\tau h}\|^2_{L^2(\ImOm)}+\frac{1}{2}\|[\na A_h^{-1} \u_{\tau h}]_{m-1}\|^2_{L^2(\Om)},
\end{aligned}
$$
where we used that $[\na A_h^{-1}\u_{\tau h}]_{m-1}$ is constant in time on $I_m$.
Thus, we obtain
$$
\tau_m^{-1}\|[\na A_h^{-1} \u_{\tau h}]_{m-1}\|^2_{L^2(\Om)}\le \|\na A_h^{-1}\mathbb{P}_h\f+\na \u_{\tau h}+\na A_h^{-1}\partial_t \u_{\tau h}\|^2_{L^2(\ImOm)}.
$$
Summing over $m$ and
 using \eqref{eq: estimate for grad} and \eqref{eq: estimate for partial_t Delta}
\begin{equation}\label{eq: estimates for jumps}
\sum_{m=1}^M\tau_m^{-1}\|[\na A_h^{-1} \u_{\tau h}]_{m-1}\|^2_{L^2(\Om)}\le C\left( \|\na A_h^{-1} \mathbb{P}_h\f\|^2_{L^2(\IOm)}+\|\mathbb{P}_h\u_0\|^2_{L^2(\Om)}\right).
\end{equation}
Combining \eqref{eq: estimate for grad}, \eqref{eq: estimate for partial_t Delta}, and \eqref{eq: estimates for jumps}, we obtain the result.
\end{proof}

\section{Stationary Stokes results }\label{sec:stokes}
For the best approximation type error estimates in the next section 
, we need to review basic results on  stationery Stokes projection and on discrete Leray  projection. 

\subsection{Discrete Stokes projection}
 In addition, we need to introduce an analogue of the Ritz projection for the stationary Stokes problem  $(R_h^S(\w,\varphi), R_h^{S,p}(\w,\varphi))\in \X_h \times M_h$ of $(\w,\varphi)\in H^1_0(\Omega)^d\times L^2(\Omega)$ given by the relation \sbox0{\ref{chap:IS:eq:stokes_projection_1,chap:IS:eq:stokes_projection_2}}
\begin{equation} \label{chap:IS:eq:stokes_projection}
\begin{aligned}
    (\nabla (\w - R_h^S(\w,\varphi)),\nabla \v_h) - (\varphi-R_h^{S,p}(\w,\varphi), \nabla \cdot \v_h) &= 0, \qquad \forall \v_h \in \X_h,\\
    (\nabla \cdot (\w - R_h^S(\w,\varphi)), q_h) &=0, \qquad \forall q_h \in M_h.
\end{aligned}
\end{equation}
\begin{remark}\label{remark:RitzStokesProjection}
If $\w$ is discrete divergence free, i.e., $(\nabla\cdot\w,q_h)=0$ for all $q_h\in M_h$, then we have $R_h^S(\w,\varphi) \in \V_h$. We will use this projection operator only for such $\w$. In this case the same projection operator is defined, e.g., in~\cite{Girault:2015}.
\end{remark}

\begin{proposition}\label{Stability_Stokes}
The Stokes projection is stable in $H^1$ norm, i.e.
$$
\|\na R_h^S(\w,\varphi)\|_{L^2(\Om)}\le C\left(\|\na\w\|_{L^2(\Om)}+\|\varphi\|_{L^2(\Om)}\right).
$$
and if $\u\in H^2(\Om)^d$ and $p\in H^1(\Om)$, the following error estimate holds
	\[
	\norm{\na (\u - R_h^S (\u,p))}_{L^2(\Omega)} \le C h \left(\norm{\nabla^2 \u}_{L^2(\Omega)} + \norm{\nabla p}_{L^2(\Omega)} \right).
	\]
\end{proposition}

\begin{proposition}\label{prop: L2 estimate}
	If $\Om$ is convex, then
	\[
	\norm{\u - R_h^S (\u,p)}_{L^2(\Omega)} \le C h \left(\norm{\nabla \u}_{L^2(\Omega)} + \norm{ p}_{L^2(\Omega)} \right)
	\]
	and
	\[
	\norm{\u - R_h^S (\u,p)}_{L^2(\Omega)} \le C h^2 \left(\norm{\nabla^2 \u}_{L^2(\Omega)} + \norm{\nabla p}_{L^2(\Omega)} \right).
	\]
\end{proposition}
 For the proofs of the above results we refer to \cite{2004Ern}, Propositions 4.14 and 4.18.

\subsection{Discrete Leray projection}
We need the restriction on the mesh and $H^2$ regularity such that the Leray projection $\mathbb{P}$ is stable in $H^1(\Om)$ norm. Such result is shown for  quasi-uniform meshes in \cite{Chrysafinos_Walkington_2010}.
\begin{lemma}[Stability of the Leray projection in $H^1$ norm]\label{lem: stability Leray in H1}
Let $\Om$ be convex and the triangulations $\{\Th\}$ be quasi-uniform. Then, there exist a constant $C$ such that
$$
\|\na \mathbb{P}_h \v\|_{L^2(\Om)}\le C\|\na \v\|_{L^2(\Om)}\quad \forall \v\in \V^1.
$$
\end{lemma}
\begin{proof}
The proof of this results is given in the Appendix of \cite{Chrysafinos_Walkington_2010}. For the completeness, we repeat the argument here. Let $\v\in \V^1$ and consider $\v_h=R_h^S (\v,0)\in \V_h$. Then by Proposition \ref{Stability_Stokes}
$$
\|\na \v_h\|_{L^2(\Om)}\le C\|\na \v\|_{L^2(\Om)}+\|0\|_{L^2(\Om)}=C\|\na \v\|_{L^2(\Om)}.
$$
On the other hand by Proposition \ref{prop: L2 estimate},
$$
\|\v- \v_h\|_{L^2(\Om)}\le Ch\|\na \v\|_{L^2(\Om)}.
$$
As a result, using the inverse and the above two estimates, we obtain
$$
\begin{aligned}
\|\na \mathbb{P}_h \v\|_{L^2(\Om)}&\le \|\na (\mathbb{P}_h \v-\v_h)\|_{L^2(\Om)}+ \|\na \v_h\|_{L^2(\Om)}\\
&\le Ch^{-1}\|\mathbb{P}_h \v-\v_h\|_{L^2(\Om)}+ C\|\na \v\|_{L^2(\Om)}\\
&\le Ch^{-1}\|\mathbb{P}_h (\v-\v_h)\|_{L^2(\Om)}+ C\|\na \v\|_{L^2(\Om)}\\
&\le Ch^{-1}\|\v-\v_h\|_{L^2(\Om)}+ C\|\na \v\|_{L^2(\Om)}\\
&\le C\|\na\v\|_{L^2(\Om)}.
\end{aligned}
$$
\end{proof}

\begin{remark}
The stability of the discrete Leray projection is a delicate matter. 
As operator, $\mathbb{P}_h$ is well-defined for $L^2(\Om)^d$ and even $L^1(\Om)^d$ functions. In such a case, the best stability result can be obtained in fractional norms (cf. \cite[~Lemma~3.1]{Guermond_2008}
\begin{equation}   \label{eq: stability Leray}
\|\mathbb{P}_h \v\|_{H^s(\Om)}\le C\| \v\|_{H^s(\Om)}\quad \forall s\in[0,\frac{1}{2}),\quad \forall \v\in H^s(\Om)^d,
\end{equation}
and can not be extended to any $s\geq 1/2$ (cf. \cite[~Remark~3.1]{Guermond_2008}.
\end{remark}

\begin{remark}
The stability of the Leray projection \eqref{eq: stability Leray} is the only result that requires restriction on a mesh.
\end{remark}

\section{Main results}  \label{sec: main results}

Now we state our main best approximation results.

\subsection{$L^2$ error estimates}

The first result establishes the  optimal error estimates $L^2(\IOm)$ norm.

\begin{theorem}\label{chap:IS:theorem:L2_best_approximation}
	Let $\f \in L^2(I;L^2(\Omega)^d)$  and $\u_0 \in \V^1$.
    Let $(\u,p)$ be the solution of \eqref{eq: weak stokes with pressure} and $(\u_{{\tau}h},p_{\tau h})$ solve the respective finite element problem \eqref{eq:spacetime_discretization}. Then, for any $\vec{\chi}\in X^w_{\tau}(\V_h)$, there holds
    \begin{equation}
	\norm{\u -\u_{{\tau}h}}_{L^2(\IOm)} \leq C \Big(\norm{\u-\vec{\chi}}_{L^2(\IOm)} + \norm{\u-\pi_\tau\u}_{L^2(\IOm)}+ \norm{\u - R_h^S \u}_{L^2(\IOm)}\Big).
    \end{equation}
\end{theorem}
\begin{proof}
	The proof essentially follows by a duality argument and Theorem \ref{thm: stability in for Ah}.
 
     Consider the following dual  problem
    \begin{equation}
	\begin{aligned}
	    -\partial_t \g(t,\x) - \Delta \g(t,\x) + \nabla \lambda(t,\x) &=\u_{{\tau}h} , \qquad & (t,\x) \in I\times \Omega,\\
	    \nabla \cdot \g(t,\x) &= 0, \qquad &(t,\x) \in I\times \Omega,\\
	    \g(t,\x) &= 0, \qquad &(t,\x) \in I\times \partial\Omega,\\
	    \g(T,x) &= 0, \qquad & \x \in \Omega.
	\end{aligned}
    \end{equation}
     The corresponding finite element approximation $(\g_{{\tau}h}, \lambda_{{\tau}h}) \in X^w_{\tau}(\X_h\times M_h)$ is given by
    \begin{equation}
    B((\v_{{\tau}h},q_{{\tau}h}),(\g_{{\tau}h},\lambda_{{\tau}h})) = \left(\u_{\tau h} ,\v_{{\tau}h}\right)_{I \times \Omega} \qquad \forall(\v_{{\tau}h},q_{{\tau}h})\in X^w_{\tau}(\X_h\times M_h).
    \end{equation}
By the Galerkin orthogonality from \eqref{prop:Galerkin_orthogonality} we have
$$
    \begin{aligned}
	\|\u_{\tau h}\|^2_{L^2(\IOm)} 
	&= (\u_{{\tau}h}, \u_{{\tau}h})_{\IOm}
			      = B((\u_{{\tau}h},p_{{\tau}h}),(\g_{{\tau}h},\lambda_{{\tau}h})) = B((\u,p),(\g_{{\tau}h},\lambda_{{\tau}h})) \\
			      &= -\sum_{m=1}^M(\u,\partial_t\g_{{\tau}h})_{I_m \times \Omega} + (\nabla \u, \nabla \g_{{\tau}h})_{I\times \Omega} - (p, \nabla \cdot \g_{{\tau}h}) - \sum_{m=1}^M(\u_m^-,[\g_{{\tau}h}]_m)_{\Omega}\\
			      &= J_1 + J_2 + J_3 + J_4,
    \end{aligned}
$$
    where we have used the dual representation of the bilinear form $B$ from \eqref{eq:dualB}.
    In the last sum we set $\g_{{\tau}h,M+1}=0$ so that $[\g_{{\tau}h}]_M= - \g_{{\tau}h,M}$.   
    Applying the Cauchy-Schwarz inequality and Theorem \ref{thm: stability in for Ah}, we obtain   
$$
\begin{aligned}
	J_1 &\le \sum_{m=1}^M \norm{\u}_{L^2(\ImOm)} \norm{\partial_t \g_{{\tau}h}}_{L^2(\ImOm)}\\
	&\le \left(\sum_{m=1}^M\norm{\u}^2_{L^2(\ImOm)} \right)^{1/2}\left(\sum_{m=1}^M \norm{\partial_t \g_{{\tau}h}}^2_{L^2(\ImOm)}\right)^{1/2}\le C\norm{\u}_{L^2(\IOm)} \norm{\u_{\tau h}}_{L^2(\IOm)}.
	\end{aligned}
$$
To treat $J_4$, we use the projection defined in \eqref{eq: projection pi_k}, the Cauchy-Schwarz inequality and the inverse inequality for time discrete function and  to obtain
$$
\begin{aligned} 
	J_4 &=-\sum_{m=1}^M(\u_m^-,[\g_{{\tau}h}]_m)_{\Omega}=-\sum_{m=1}^M((\pi_\tau\u)_m,[\g_{{\tau}h}]_m)_{\Omega} \\
	&\le \left(\sum_{m=1}^{M} \tau_m\|(\pi_\tau\u)_m\|^2_{L^2(\Om)}\right)^{1/2}\left(\sum_{m=1}^{M} \tau^{-1}_m\|[\g_{{\tau}h}]_{m-1}\|^2_{L^2(\Om)}\right)^{1/2}\\
	&\le \left(\sum_{m=1}^{M} \tau_m\|\pi_\tau\u\|^2_{L^\infty(I_m;L^2(\Om))}\right)^{1/2}\left(\sum_{m=1}^{M} \tau_m\|\tau_m^{-1}[\g_{{\tau}h}]_{m-1}\|^2_{L^2(\Om)}\right)^{1/2}\\
	&\le C\left(\sum_{m=1}^{M} \|\pi_\tau\u\|^2_{L^2(\ImOm)}\right)^{1/2}\left(\sum_{m=1}^{M} \tau_m\|\tau_m^{-1}[\g_{{\tau}h}]_{m-1}\|^2_{L^2(\Om)}\right)^{1/2}\\
	& \le C\norm{\pi_\tau\u}_{L^2(\IOm)} \norm{\u_{\tau h}}_{L^2(\IOm)}.
\end{aligned} 
$$
    For $J_2+J_3$ we can argue by using the projection $R_h^S$ defined in \eqref{chap:IS:eq:stokes_projection}.
    Then, we have
    $$
    \begin{aligned}
	J_2 + J_3 &=(\nabla \u, \nabla \g_{{\tau}h})_{I\times \Omega} - (p,\nabla \cdot \g_{{\tau}h})_{I\times\Omega}\\
		  &=(\nabla R_h^S(\u,p), \nabla \g_{{\tau}h})_{I\times \Omega} - (R_h^{S,p}(\u,p),\nabla \cdot \g_{{\tau}h})_{I\times\Omega}\\
		  &
		  =(\nabla R_h^S(\u,p), \nabla \g_{{\tau}h})_{I\times \Omega},
    \end{aligned}
    $$
    where the last term vanishes, since $\g_{{\tau}h}$ is discretely divergence-free. Here and in what follows, the projection $(R_h^S,R_h^{S,p})$ is applied to time dependent functions $(\u,p)$ pointwise in time. Since $\nabla \cdot \u(t) = 0$ for almost all $t \in I$ we have $R_h^S(\u(t),p(t)) \in \V_h$, cf. \eqref{remark:RitzStokesProjection}.
    With this we can use the definition of the discrete Stokes operator $A_h$ resulting in 
    \begin{equation}\label{eq:est_R_h_g_h}
    \begin{aligned}
	(\nabla R_h^S(\u,p), \nabla \g_{{\tau}h})_{I\times \Omega} 
	&= (R_h^S(\u,p), A_h \g_{{\tau}h})_{I\times \Omega}\\
	&\le  \norm{R_h^S (\u,p)}_{L^2(\IOm)}\norm{A_h\g_{{\tau}h}}_{L^2(\IOm)}\\
	&\le  \left(\norm{\u}_{L^2(\IOm))} + \norm{\u - R_h^S (\u,p)}_{L^2(\IOm)}\right)\norm{\u_{\tau h}}_{L^2(\IOm)}.
    \end{aligned}
    \end{equation}
    Combining the estimates for $J_1$, $J_2$, $J_3$ and $J_4$, we conclude    
  $$
	\norm{\u_{\tau h}}_{L^2(\IOm)}
	\le C \Big(\norm{\u}_{L^2(\IOm))} + \norm{\pi_\tau\u}_{L^2(\IOm))}+  \norm{\u - R_h^S(\u,p)}_{L^2(\IOm)}\Big).
$$
Using  that the Galerkin method is invariant on $X^w_{\tau}(\V_h\times M_h)$, by replacing $\u$ and $\u_{\tau h}$ with $\u-\vec{\chi}$ and $\u_{\tau h}-\vec{\chi}$ for any $\vec{\chi}\in X^w_{\tau}(\V_h)$, and using the triangle inequality  we complete the proof of the theorem.
\end{proof}

If the exact solution is sufficiently smooth then the above result easily leads to optimal convergence rates.
\begin{corollary}\label{cor: optimal error estimate in L2}
\label{theorem:error_estimate}
Let $\Omega$ be convex, $\f \in L^2(I,L^2(\Omega)^d)$ and $\u_0 \in \V^1$. Let $(\u,p)$ be the solution of \eqref{eq: weak stokes with pressure} and $(\u_{{\tau}h},p_{\tau h})$ solve the respective finite element problem \eqref{eq:spacetime_discretization}. Then, there holds
    \begin{equation}
\norm{\u-\u_{{\tau}h}}_{L^{2}(I;L^2(\Omega))}
\leq C \left({\tau} + h^2\right) \left( \norm{\f}_{L^{2}(I;L^2(\Omega))} + \norm{\u_0}_{\V^1}\right).
\end{equation}
\end{corollary}
\begin{proof}
	From Theorem \ref{chap:IS:theorem:L2_best_approximation}, we need to estimate three terms.	
 The temporal error is estimated by \eqref{eq: approximation of pi_k} resulting in
	\[
	\norm{\u - \pi_\tau \u}_{L^{2}(I;L^2(\Omega))} \le C \tau \norm{\partial_t \u}_{L^2(I;L^2(\Omega))}.
	\]
	The spatial error using Proposition \ref{prop: L2 estimate} is estimated by
	\[
	\norm{\u - R_h^S (\u,p)}_{L^2(I;L^2(\Omega))} \le C h^2 \left(\norm{\nabla^2 \u}_{L^2(I;L^2(\Omega))} + \norm{\nabla p}_{L^2(I;L^2(\Omega))} \right).
	\]
	Similarly, choosing $\vec{\chi}= P_\tau R_h^S (\u,p) $, where $P_\tau$ is the orthogonal $L^2$ projection onto $X^w_{\tau}$,   by the triangle inequality we obtain
	$$
	\begin{aligned}
	\norm{\u-\vec{\chi}}_{L^2(\IOm)}&\le \norm{\u-P_\tau\u}_{L^2(\IOm)}+\norm{P_\tau(\u-R_h^S (\u,p) }_{L^2(\IOm)}\\
	&\le    C \tau \norm{\partial_t \u}_{L^2(I;L^2(\Omega))}+C h^2 \norm{\nabla^2 \u}_{L^2(I;L^2(\Omega))}.
	\end{aligned}
	$$
	Using Corollary \ref{cor:Omega_convex_2}, we obtain the result. 
\end{proof}


\subsection{$H^1$ error estimates} \label{subsec: H1 results}

The second result establishes the  optimal error estimates $L^2(I; H^1(\Om)^d)$ norm on quasi-uniform meshes. 

\begin{theorem}\label{thm: H1 best approximation}
Let $\f \in L^2(I;L^2(\Omega)^d)$  and $\u_0 \in \V^1$.
    Let $(\u,p)$ be the solution of \eqref{eq: weak stokes with pressure} and $(\u_{{\tau}h},p_{\tau h})$ solve the respective finite element problem \eqref{eq:spacetime_discretization} on a family of quasi-uniform triangulations $\{\Th\}$.
Then, for any $\vec{\chi}\in X^w_{\tau}(\V_h)$, there holds
    \begin{equation}
	\norm{\na (\u -\u_{{\tau}h})}_{L^2(\IOm)} \leq C \Big(\norm{\na(\u-\vec{\chi})}_{L^2(\IOm)} + \norm{\na(\u-\pi_\tau\u)}_{L^2(\IOm)}+ \norm{\na(\u - R_h^S \u)}_{L^2(\IOm)}\Big).
    \end{equation}
\end{theorem}

\begin{proof}
	The proof is similar to the proof for the $L^2(I;L^2(\Om)^d)$ norm. This time, we consider the following dual  problem
    \begin{equation}
	\begin{aligned}
	    -\partial_t \g(t,\x) - \Delta \g(t,\x) + \nabla \lambda(t,\x) &=A_h\u_{{\tau}h} , \qquad & (t,\x) \in I\times \Omega,\\
	    \nabla \cdot \g(t,\x) &= 0, \qquad &(t,\x) \in I\times \Omega,\\
	    \g(t,\x) &= 0, \qquad &(t,\x) \in I\times \partial\Omega,\\
	    \g(T,x) &= 0, \qquad & \x \in \Omega.
	\end{aligned}
    \end{equation}
     The corresponding finite element approximation $(\g_{{\tau}h}, \lambda_{{\tau}h}) \in X^w_{\tau}(\X_h\times M_h)$ is given by
    \begin{equation}
    B((\v_{{\tau}h},q_{{\tau}h}),(\g_{{\tau}h},\lambda_{{\tau}h})) = \left(A_h\u_{\tau h} \,\v_{{\tau}h}\right)_{I \times \Omega} \qquad \forall(\v_{{\tau}h},q_{{\tau}h})\in X^w_{\tau}(\X_h\times M_h).
    \end{equation}
By the Galerkin orthogonality from \eqref{prop:Galerkin_orthogonality} we have
$$
    \begin{aligned}
	\|\na \u_{\tau h}&\|^2_{L^2(\IOm)} 
	= (A_h\u_{{\tau}h}, \u_{{\tau}h})_{\IOm}
			      = B((\u_{{\tau}h},p_{{\tau}h}),(\g_{{\tau}h},\lambda_{{\tau}h})) = B((\u,p),(\g_{{\tau}h},\lambda_{{\tau}h})) \\
			      &= -\sum_{m=1}^M(\u,\partial_t\g_{{\tau}h})_{I_m \times \Omega} + (\nabla \u, \nabla \g_{{\tau}h})_{I\times \Omega} - (p, \nabla \cdot \g_{{\tau}h}) - \sum_{m=1}^M(\u_m^-,[\g_{{\tau}h}]_m)_{\Omega}\\
			      &= J_1 + J_2 + J_3 + J_4,
    \end{aligned}
$$
    where we have used the dual representation of the bilinear form $B$ from \eqref{eq:dualB}.
    In the last sum we set $\g_{{\tau}h,M+1}=0$ so that $[\g_{{\tau}h}]_M= - \g_{{\tau}h,M}$.   
    Applying the Cauchy-Schwarz inequality and Theorem \ref{thm: stability in for grad_Ah}, we obtain   
$$
\begin{aligned}
	J_1 & =  -\sum_{m=1}^M(\mathbb{P}_h\u,\partial_t\g_{{\tau}h})_{I_m \times \Omega}\\
	&=-\sum_{m=1}^M(\na \mathbb{P}_h\u,\na A_h^{-1}\partial_t\g_{{\tau}h})_{I_m \times \Omega}=\\
	&\le \sum_{m=1}^M \norm{\na \mathbb{P}_h\u}_{L^2(\ImOm)} \norm{\na A_h^{-1}\partial_t \g_{{\tau}h}}_{L^2(\ImOm)}\\
	&\le \left(\sum_{m=1}^M\norm{\na \mathbb{P}_h\u}^2_{L^2(\ImOm)} \right)^{1/2}\left(\sum_{m=1}^M \norm{\na A_h^{-1}\partial_t \g_{{\tau}h}}^2_{L^2(\ImOm)}\right)^{1/2}\\
	&\le C\norm{\na \u}_{L^2(\IOm)} \norm{\na\u_{\tau h}}_{L^2(\IOm)},
	\end{aligned}
$$
where in the last step we used the stability of $\mathbb{P}_h$ in $H^1(\Om)^d$ for divergence free functions from $H^1_0(\Om)^d$. 

To treat $J_4$, we use the projection defined in \eqref{eq: projection pi_k}, the Cauchy-Schwarz inequality and the inverse inequality for time discrete function and  to obtain
$$
\begin{aligned} 
	J_4 &=-\sum_{m=1}^M(\u_m^-,[\g_{{\tau}h}]_m)_{\Omega}=-\sum_{m=1}^M(\pi_\tau\mathbb{P}_h\u_m,[\g_{{\tau}h}]_m)_{\Omega} \\
	&=-\sum_{m=1}^M(\pi_\tau\na\mathbb{P}_h\u_m,[\na A_h^{-1}\g_{{\tau}h}]_m)_{\Omega}\\
	&\le \left(\sum_{m=1}^{M} \tau_m\|\pi_\tau\na\mathbb{P}_h\u_m\|^2_{L^2(\Om)}\right)^{1/2}\left(\sum_{m=1}^{M} \tau^{-1}_m\|[\na A_h^{-1}\g_{{\tau}h}]_{m-1}\|^2_{L^2(\Om)}\right)^{1/2}\\
	&\le \left(\sum_{m=1}^{M} \tau_m\|\pi_\tau\na\mathbb{P}_h\u\|^2_{L^\infty(I_m;L^2(\Om))}\right)^{1/2}\left(\sum_{m=1}^{M} \tau_m\|\tau_m^{-1}[\na A_h^{-1}\g_{{\tau}h}]_{m-1}\|^2_{L^2(\Om)}\right)^{1/2}\\
	&\le C\left(\sum_{m=1}^{M} \|\pi_\tau\na \mathbb{P}_h\u\|^2_{L^2(\ImOm)}\right)^{1/2}\left(\sum_{m=1}^{M} \tau_m\|\tau_m^{-1}[\na A_h^{-1}\g_{{\tau}h}]_{m-1}\|^2_{L^2(\Om)}\right)^{1/2}\\
	& \le C\norm{\pi_\tau\na \mathbb{P}_h\u}_{L^2(\IOm)} \norm{\na\u_{\tau h}}_{L^2(\IOm)}\\
	& \le C\norm{\pi_\tau\na \u}_{L^2(\IOm)} \norm{\na\u_{\tau h}}_{L^2(\IOm)},
\end{aligned} 
$$
where again in the last step we used the stability of $\mathbb{P}_h$ in $H^1(\Om)^d$ for divergence free functions from $H^1_0(\Om)^d$. 

    For $J_2+J_3$ we can argue by using the projection $R_h^S$ defined in \eqref{chap:IS:eq:stokes_projection}.
    Then since $\g_{{\tau}h}$ is discretely divergence-free, we have
    $$
    \begin{aligned}
	J_2 + J_3 &=(\nabla \u, \nabla \g_{{\tau}h})_{I\times \Omega} - (p,\nabla \cdot \g_{{\tau}h})_{I\times\Omega}\\
		  &=(\nabla R_h^S(\u,p), \nabla \g_{{\tau}h})_{I\times \Omega} - (R_h^{S,p}(\u,p),\nabla \cdot \g_{{\tau}h})_{I\times\Omega}
		  =(\nabla R_h^S(\u,p), \nabla \g_{{\tau}h})_{I\times \Omega}\\
		  &\le  \norm{\na R_h^S (\u,p)}_{L^2(\IOm)}\norm{\na\g_{{\tau}h}}_{L^2(\IOm)}\\
		\le & C\left(\norm{\na\u}_{L^2(\IOm))} + \norm{\na(\u - R_h^S (\u,p))}_{L^2(\IOm)}\right)\norm{\na\u_{\tau h}}_{L^2(\IOm)}.
    \end{aligned}
      $$
       Combining the estimates for $J_1$, $J_2$, $J_3$ and $J_4$, we conclude    
  $$
	\norm{\na \u_{\tau h}}_{L^2(\IOm)}
	\le C \Big(\norm{\na\u}_{L^2(\IOm))} + \norm{\pi_\tau\na\u}_{L^2(\IOm))}+  \norm{\na(\u - R_h^S(\u,p))}_{L^2(\IOm)}\Big).
$$
Using  that the Galerkin method is invariant on $X^w_{\tau}(\V_h\times M_h)$, by replacing $\u$ and $\u_{\tau h}$ with $\u-\vec{\chi}$ and $\u_{\tau h}-\vec{\chi}$ for any $\vec{\chi}\in X^w_{\tau}(\V_h)$, and using the triangle inequality  we complete the proof of the theorem.
\end{proof}

\begin{corollary}\label{cor: optimal error estimate in H1}
Let $\Omega$ be convex, $\f \in L^2(I,L^2(\Omega)^d)$ and $\u_0 \in \V^1$. Let $(\u,p)$ be the solution of \eqref{eq: weak stokes with pressure} and $(\u_{{\tau}h},p_{\tau h})$ solve the respective finite element problem \eqref{eq:spacetime_discretization}. Then, there holds
    \begin{equation}
\norm{\na(\u-\u_{{\tau}h)}}_{L^{2}(I;L^2(\Omega))}
\leq C \left({\tau}^{1/2} + h\right) \left( \norm{\f}_{L^{2}(I;L^2(\Omega))} + \norm{\u_0}_{\V^1}\right).
\end{equation}
\end{corollary}
\begin{proof}
The proof is analogous to the proof of Corollary \ref{cor: optimal error estimate in L2}. The main difference is to use the estimate 
$$
\norm{\na(\u-\pi_\tau\u)}_{L^2(\IOm)}\le C\tau^{1/2}\left( \norm{\f}_{L^{2}(I;L^2(\Omega))} + \norm{\u_0}_{\V^1}\right).
$$
from Lemma~3.13 from \cite{KunischK_PieperK_VexlerB_2014}
\end{proof}

\section{Inhomogeneous heat equation} \label{sec: parabolic}

The above results naturally hold for the scalar parabolic equation 
\begin{equation}\label{eq: heat equation}
\begin{aligned}
u_t(t,x)-\Delta u(t,x) &= f(t,x), & (t,x) &\in \IOm,\;  \\
    u(t,x) &= 0,    & (t,x) &\in I\times\pa\Omega, \\
   u(0,x) &= u_0(x),    & x &\in \Omega.
\end{aligned}
\end{equation}
All arguments stay almost unchanged.
We only need to replace $A$ and $A_h$ with $-\De$ and $-\De_h$, the discrete Leray projection $\mathbb{P}_h$ with the $L^2$-projection ${P}_h$, and the Stokes projection $R_h^S(\u,p)$ with the Ritz (elliptic) projection $R_h$.   The mesh restrictions in the case of $H^1$ norm estimates can be relaxed. The only technical requirement is the stability of the $L^2$-projection in $H^1$ norm, i.e.
$$
\|\na P_h u\|_{L^2(\Om)}\le C\|\na u\|_{L^2(\Om)}.
$$
Such result is shown for locally quasi-uniform meshes \cite{BrambleJH_PasciakJE_SteinbachO_2002a}, on adaptive meshes obtained by bisection method in 2D  \cite{BankRE_YserentantH_2014, GaspozFD_HeineCJ_SiebertKG_2016}.

In this situation the best approximation result take the form
\begin{equation}\label{eq: optimal error L2 heat}
\|u-u_{kh}\|_{L^2(I\times \Om)}\le C\left(\|u-\chi\|_{L^2(I\times\Om)}+\|R_hu-u\|_{L^2(I\times\Om)}+\|\pi_ku-u\|_{L^2(I\times\Om)}\right),
\end{equation}
and since the Ritz projection $R_h$ is stable in $H^1$ norm,
\begin{equation}\label{eq: optimal error H1 heat}
\|\na(u-u_{kh})\|_{L^2(I\times \Om)}\le C\left(\|\na(u-\chi)\|_{L^2(I\times\Om)}+\|\na(\pi_ku-u)\|_{L^2(I\times\Om)}\right),
\end{equation}
 for any space-time fully discrete function $\chi$. Similarly, to the Stokes problem, under the additional assumption of convexity of $\Omega$ and some approximation properties of the discrete spaces,
we easily derive optimal error estimates of the form
\begin{equation}\label{eq: error in terms of data heat}
\norm{u-u_{{\tau}h}}_{L^{2}(I;H^s(\Omega))}
\leq C \left({\tau}^{1-s/2} + h^{2-s}\right) \left( \norm{f}_{L^{2}(I;L^2(\Omega))} + \norm{u_0}_{H^1(\Om)}\right),\quad s=0,1.
\end{equation}

\bibliographystyle{siam}
\bibliography{sources}

\end{document}